\documentclass[11pt, reqno]{amsart}
\usepackage{amscd, color, bbm}
\usepackage{url}
\usepackage{parskip}
\usepackage{fancyhdr}
\usepackage{vmargin}
\usepackage{appendix}
\usepackage{graphicx, amsmath, amssymb, amscd, amsthm, euscript, psfrag, amsfonts,bm,xcolor}
\usepackage{etoolbox}
\usepackage{multicol}
\patchcmd{\thebibliography}{\chapter*}{\section*}{}{}
\setmarginsrb{2 cm}{2 cm}{2 cm}{2 cm}{1 cm}{1 cm}{1 cm}{1 cm}

\newtheorem{theorem}{Theorem}[section]
\newtheorem{lemma}[theorem]{Lemma}
\newtheorem{proposition}[theorem]{Proposition}
\newtheorem{corollary}[theorem]{Corollary}

\newtheorem{claim}[theorem]{Claim}
\theoremstyle{definition}
\newtheorem{definition}[theorem]{Definition}

\newtheorem{example}[theorem]{Example}
\theoremstyle{remark}
\newtheorem{remark}[theorem]{Remark}

\newtheorem*{ack}{Acknowledgements}

\newcommand{\N}{\mathbb{N}}

\newcommand{\R}{\mathbb{R}}

\newcommand{\ra}{\rightarrow}

\DeclareMathAlphabet{\mathpzc}{OT1}{pzc}{m}{it}

\newcommand{\norm}[1]{\lVert #1 \rVert}
\numberwithin{equation}{section}

\DeclareMathAlphabet{\mathpzc}{OT1}{pzc}{m}{it}

\title{A typical number is extremely non-normal}
\author{Anastasios Stylianou}

\address{Mathematics Institute,
Zeeman Building,
University of Warwick,
Coventry CV4 7AL}
\email{Tasos.Stylianou@warwick.ac.uk}

\begin{document}
%%%%%%%%%%%%%%%%%%%%%%%%%
\begin{abstract}
Fix a positive integer $N\geq2$. For a real number $x\in[0,1]$ and a digit $i\in\{0, 1,...,N-1\}$, let $\Pi_i(x, n)$ denote the frequency of the digit $i$ among the first $n$ $N$-adic digits of $x$. It is well-known that for a typical (in the sense of Baire) $x\in[0, 1]$, the sequence of digit frequencies  diverges as $n\ra\infty$. In this paper we show that for any regular linear transformation $T$ there exists a residual set of points $x\in[0, 1]$ such that the $T$-averaged version of the sequence $(\Pi_i(x, n))_n$ also diverges significantly.
\end{abstract}
%%%%%%%%%%%%%%%%%%%%%%%%%%%
\maketitle
%%%%%%%%%%%%%%%%%%%%%%%%%%%
\section{Statement of results}
%%%%%%%%%%%%%%%%%%%%%%%%%%%%%%%%%%%5
Fix a positive integer $N\geq2$. Throughout this paper, we will consider the unique, non-terminating, base $N$ expansion of a number $x\in[0, 1]$, written as $x = \frac{d_1(x)}{N}+ \frac{d_2(x)}{N^2} + ... + \frac{d_n(x)}{N^n}+ ...$ with $ d_i(x)\in\{0,1,...N-1\}$. For each digit $i\in\{0,1,...,N-1\}$, we will write \[ 
\Pi_i(x, n) :=\frac{1}{n}\left|\left\{   1\leq j\leq n\; :\; d_j(x)=i    \right\}\right|
\]
for the frequency of the digit $i$ among the first $n$ digits of $x$. Let $\Delta_N$ denote the family of $N$-dimensional probability vectors, that is
\[
\Delta_N :=\left\{\left (p_0, p_1,...,p_{N-1}\right)\; :\; p_i\geq 0,\;\sum_{i=0}^{N-1}p_i = 1  \right\}
\]
and write
$$
 \mathbf{\Pi}(x, n) = \left(\Pi_0(x, n),\Pi_1(x, n),...,\Pi_{N-1}(x, n)\right)
$$
for the vector of frequencies of digits. Then clearly for every natural number $n$ the vector of frequencies $\mathbf \Pi(x, n)$ lies in the simplex $\Delta_N$, and consequently, the set of accumulation points of the sequence $\left(\mathbf\Pi(x, n)\right)_n$ is a subset of $\Delta_N$.

Borel's famous normal number theorem states that for Lebesgue almost every point in $[0,1]$ its vector of frequencies converges towards the uniform probability vector $\left(\frac{1}{N},...,\frac{1}{N} \right)$ with respect to $\norm{.}_1$ as $n\ra\infty$. More generally Birkoff's Ergodic Theorem tells us that for each probability measure $\mu$ of $[0,1]$ which is invariant under the transformation $x\mapsto Nx \mod1$, the sequence $\mathbf \Pi(x, n)$ converges to a specific probability vector $p_\mu\in\Delta_N$ for $\mu$ almost every $x\in[0,1]$ as $n\ra\infty$. Yet still from a topological viewpoint there is a completely different picture.

We recall that in a metric space $X$, a set $R$ is called residual if its complement is of the first category. Also, recall that we say that a typical element $x\in X$ has property $P$ if the set $ R :=\{ x \in X \; :\;  x \text{ has property }P\}$ is residual. We refer the reader to Oxtoby \cite{Ox} for more details. Recently, the study of the limiting behaviour of the sequence of frequencies $\left(\Pi_i(x,n)\right)_n$ from a topological viewpoint has enjoyed a great interest from many authors \cite{Vo, Sa1, APT, CZ, Ol, Sa2, Si, SS, SS2}. In particular, perhaps unexpectedly, Olsen showed that for a typical $x\in[0, 1]$, the set of accumulation points of $(\mathbf\Pi(x, n))_n$ is all of $\Delta_N$.

\begin{theorem}[A, \cite{Ol}]\label{olsen} 
The set \[
\left\{x \in [0, 1]\;:\; \text{The set of the accumulation points of $(\mathbf\Pi(x, n))_n$ equals $\Delta_N$}\right\}
\]
is residual.
\end{theorem}

The result above has inspired a series of papers (\cite{Ol2, Ol3, Ol4, MAN, MAN2}) that appeared later in the literature. These later results were sometimes considering different settings such us continued fraction expansions, iterated function systems, finite or countable Markov partitions and other times enhancing the way in which the vectors of frequencies were diverging.
In this paper we offer a generalisation of the second form. Namely, Theorem \ref{main} states that for any regular linear transformation $T$, there exists a residual set of points $x$ whose set of accumulation points of their $T$-averaged vector of frequencies $\left(\Pi^T(x, m)\right)_m$ includes the simplex of all $N$-dimensional probability vectors. The precise definitions are given below.

Given a divergent sequence, by considering an averaged version of this sequence one may succeed in producing a convergent sequence. In this paper we prove a somewhat unexpected result that strengthens Theorem \ref{olsen} considerably. To do that we need the following definition.
%%%%%%%%%%%%%%%%%%%%%%%%%%%%
\begin{definition}[\cite{hard}]
An $\mathbb{N}\times\mathbb{N}$ real valued matrix $T=\left[c_{m,n}\right]$, is called a \emph{linear transformation} if for any real sequence $(s_{n})_{n}$ and any natural number $m$, we have that $t_{m}=\sum_{n\geq1} c_{m,n} s_{n}$ converges to a real number. The sequence $(t_{m})_{m}$ is called the $T$-averaged version of $(s_{n})_{n}$ and if $\lim_{m\to\infty}t_{m}$ exists we call it the $T$-value of $(s_{n})_{n}$.
\end{definition}
%%%%%%%%%%%%55
\begin{example}
The linear transformation $T=\left[c_{m,n}\right]_{\mathbb{N}\times\mathbb{N}}$ with $c_{m,n}=$ $\begin{cases}
\frac{1}{m}\;\;\;\text{if $1\leq n\leq m$}\\
0\;\;\;\text{otherwise}
\end{cases}$
corresponds to considering the sequence of the natural averages of a sequence.
\end{example}
%%%%%%%%%%555

Namely, we show that for any linear transformation $T$, the $T$-averaged version of the sequence $(\mathbf\Pi(x, n))_n$ accumulates at every point in $\Delta_N$ for a residudal set of points in $[0,1]$. The special case of Cesàro averages was considered in \cite{Ol3}. To state this result a bit more precisely, we define a regular linear transformation below. 

Linear transformations are used to help generalise the sense of a limit of sequence, in order to provide a \emph{sensible} value of a limit for a divergent sequence. However, if a sequence is already convergent we would like the averaged version of this sequence to also converge to the same limit. We hence call a linear transformation $T=\left[c_{m,n}\right]$ \emph{regular}, if $T$ maps all convergent sequences to their original limit, that is $T$ is regular if
$$
t_{m}\rightarrow l\,\,\,\text{as}\,\,\,m\rightarrow\infty\;\;\text{ whenever }\;\;s_{k}\rightarrow l\,\,\,\text{as}\,\,\,k\rightarrow\infty.
$$
%%%%%%%%%%%%%%%%%%%%%%%%

Fix a regular linear transformation $T=[c_{m,n}]_{\N\times\N}$ and consider the $T$-averaged version of the sequences of digit frequencies
$$
\Pi_i^T(x,m):=\sum_{n=1}^\infty c_{m,n}\,\Pi_i(x,n)\;\;\;\;\;\text{for } 0\leq i\leq N-1.
$$
 Then, the $T$-averaged version of the vector of digit frequencies is given by the sequence with $m$th term  \[
\mathbf\Pi^T(x,m):=\left(\Pi^T_0(x, m),\Pi^T_1(x, m),...,\Pi^T_{N-1}(x, m)\right)
\]

We now state our main theorem.
\begin{theorem}\label{main}
For any regular linear transformation $T$ the set
\[
R_T:=\left \{ x\in[0, 1]\; :\; \text{The sequence $\left(\mathbf\Pi^T(x, m)\right)_m$ accumulates at every point of $\Delta_N$} \right\},
\]
is residual.
\end{theorem}

\begin{remark}
Results similar in nature but in somewhat more general settings were obtained independently by Olsen and West in \cite{Ol4} and Aveni and Leonetti in \cite{ita}.
\end{remark}

% In the following corollary we show that in fact for each point in $\mathcal{R}_\mathcal{T}$ its vector of frequencies diverges spectacularly even for limit points of the set of linear transformations $\mathcal{T}$ with respect to the supremum norm, hence further strengthening the contrast with the measure theoretical viewpoint.
% \begin{corollary}\label{coro}
% For $x\in\bigcap_{T\in \mathcal{T}}  R_{T}$ and any regular linear transformation $S$ in the boundary of $\mathcal{T}$ with respect to the supremum norm we have that $\text{the sequence $\left(\mathbf\Pi^S(x, m)\right)_m$ accumulates at every point of $\Delta_N$}$.
% \end{corollary}
 %%%%%%%%%%%%%%%%%%%%%%%
The proof of Theorem \ref{main} is given in Section \ref{proof} below. Before proceeding we make two final remarks concerning our result .
Firstly, fix a countable family of regular linear transformations $\mathcal{T}=\{T_1, T_2,...\}$ (Examples include H\"older means, Ces\`aro averages,...). Notice that it clearly follows from Theorem $\ref{main}$ that the countable intersection $\mathcal{R}_\mathcal{T}:=\bigcap_{T\in\mathcal{T}}  R_{T}$ is also residual.

Additionally, also in terms of dimensions, the size of $R_T$ varies between \emph{very big} and \emph{very small} depending on the exact viewpoint. Below we write $\dim_H$ and $\dim_P$ for the Hausdorff dimension and the packing dimension, respectively. The reader is referred to \cite{Fa} for the definitions of the Hausdorff and packing dimensions. It follows from \cite{OW} that $ \dim_H R_T = 0$. On the other hand, using Theorem \ref{main} we can find the packing dimension of the set $R_T$. A result analogue to the following corollary appeared in \cite{Ol} with a slightly inaccurate proof. We postpone our argument until the next section. 
\begin{corollary}\label{dimen}
For any regular linear transformation $T$ the set $R_T$ has full packing dimension, that is
\[\dim_P\, R_T = 1\]
\end{corollary}

% Observe that this vastly contrasts the measure theoretical viewpoint. Fistrly, Borel's Normal Number theorem says for Lebesque almost every point in $[0,1]$ its vector of frequencies converges towards the uniform probability vector $\left(\frac{1}{N},...,\frac{1}{N} \right)$ but also more generically Birkoff's Ergodic Theorem tells that for each 

% This contrasts vastly with the topological point of view. Namely, Theorem \ref{main} states that for any regular linear transformation $T$, there exists a residual set of points $x$ whose set of accumulation points of the $T$-averaged vector of frequencies $\left(\Pi^T(x, m)\right)_m$ includes the simplex of all $N$-dimensional probability vectors.
%%%%%%%%%%%%%%%%%%%%%%%%%%%%%%%%%%%%%%%%%%%%%%%%%%%%
%\section{Regular Linear Transformations}
%%%%%%%%%%%%%%%%%%%%%%%%%%%%%%

%%%%%%%%%%%%%%%%%%%%%%%%%%%%%%%%%%%%%%%5
\section{Proof of Theorem \ref{main}} \label{proof}
%%%%%%%%%%%%%%%%%%%%%%%%%%%%%%%%%%%%%%%%55

Throughout the proof, we will work with a subset of $[0,1]$, namely
\[
\mathbb{I}:= [0, 1]\setminus \left\{x\in [0, 1]\; :\; x \text{ has a terminating $N$-adic expansion}\right\}.
\]
For brevity, write \[
\mathbb{D}:=\mathbb{Q}^N\cap\Delta_N\setminus\{(1,0,...,0),(0,...,0,1)\}
\]
where we excluded two particular vectors for technical reasons, which become apparent in the proof of Claim \ref{dense}.
% %%%%%%%%%%%%%%%%%%%%%%%%%%%%%%%%%%%5

Let $\psi:\N\mapsto\N$ be any increasing function such that  $\psi(n)>n^2$ for all $n\in\N$, and write $\psi^{m}$ for the composition of $\psi$ with itself $m$ times. For $h\in\N$ we define the property $P_{h,\psi}$ as follows.
\begin{definition}
We say that a sequence $(\mathbf x_n)_n$ in $\R^N$ has property $P_{h,\psi}$ if for all $i\in\N,\;m\in\N$ and all $\mathbf q\in\mathbb{D}$, there exists $j\in\N$ satisfying the following:
\begin{itemize}
  \item $j\geq i$
  \item $\frac{j}{\psi(j)}\le \frac{1}{h}$
  \item if $j\le n\le \psi^m(j)$ then $\|\mathbf x_n-\mathbf q \|_1<\frac{1}{h}$.
\end{itemize}
\end{definition}

%%%%%%%%%%%%%%%%%%%%%%%%%%

\noindent Our proof of Theorem \ref{main} will consist of three lemmas:
\begin{enumerate}
  \item First we will prove that for any $h\in\N$ and any increasing $\psi:\N\mapsto\N$ such that  $\psi(n)>n^2$ for all $n\in\N$, the following set is residual:
  \begin{align}\label{a}
    A_{h,\psi}:=\left\{ x \in\mathbb{I}\; :\;(\mathbf\Pi(x, n))_n \text{ has property $P_{h,\psi}$}\right\}.
  \end{align} 
\item Then we will show that if $T$ is any regular linear transformation and $h\in\N$ then there exists $\psi_{h,T}:\N\mapsto\N$ such that if $x\in A_{6Mh,\psi_{h,T}}$, with $M$ depending on $T$, then  $(\mathbf\Pi^T(x,m))_{m}$ has property $P_{h,\psi_{h,T}}$.
%%%%%%%%%%%%
%%%%%%%%%%%%%%%%5
\item Finally, setting $A_T:=\bigcap_{h\in\N} A_{6Mh,\psi_{h,T}}$ we will show that \[ A_{T}\subset R_T\] (recall that $R_T$ is defined in Theorem \ref{main}).
\end{enumerate}

\begin{lemma}\label{family}
For all  $h\in\N$ and all increasing $\psi:\N\mapsto\N$ such that  $\psi(n)>n^2$ for all $n\in\N$, the set $A_{h,\psi}$ is residual (recall that $A_{h,\psi}$ is defined in (\ref{a})).
\end{lemma}

\begin{proof}
For fixed $i\in\N,\; m\in\N$ and $\mathbf q\in\mathbb D$, we define property $P_{h,\psi,i,m,\mathbf q}$, as follows. We say that a sequence $(\mathbf x_n)_n$ has property $P_{h,\psi,i,m,\mathbf q}$ if there exists $j\in\N$ satisfying
\begin{enumerate}
  \item $j\geq i$
  \item $\frac{j}{\psi(j)}<\frac{1}{h}$
  \item if $j<n<\psi^m(j)$ then $\|\mathbf x_n-\mathbf q\|_1<\frac{1}{h}$.
\end{enumerate}

Let $G_{h,\psi,i,m,\mathbf q}:=\left\{ x\in\mathbb I\; :\; (\mathbf\Pi(x, n))_n \text{ has property $P_{h,\psi,i,m,\mathbf q}$}\right\}$. Clearly,
\[
\bigcap_{i\in\N}\bigcap_{m\in\N}\bigcap_{\mathbf q\in\mathbb D}\,G_{h,\psi,i,m,\mathbf q}=A_{h,\psi}
\]

\begin{claim}\label{open}
$G_{h,\psi,i,m,\mathbf q}$ is open.
\end{claim}
\begin{proof}
Let $x\in G_{h,\psi,i,m,\mathbf q}$. Then, there exists a positive integer $j$ such that $j\geq i$, $\frac{j}{\psi(j)}< \frac{1}{h}$, and if $j < n < \psi^m(j )$, then $\left\|\mathbf\Pi(x,n) -\mathbf q \right\|_1<\frac{1}{h}$.
We now choose $\delta$ to equal $\frac{1}{N^{a}}$ where $a>\psi^{m+1}(j)+1$ and since $x\in\mathbb I$ we can also insist that $a\in\N$ is such that the $(a-1)$ and $(a-2)$-nd digits of $x$ are not both either $0$ or $N-1$. . Then all $y\in B(x,\delta)$ have their first $\psi^m(j)$ digits the same as $x$, and so $B(x,\delta)\subset G_{\psi,h,i,m,\mathbf q}$. This completes the proof of Claim \ref{open}.
\end{proof}

\begin{claim}\label{dense}
$G_{h,\psi,i,m,\mathbf q}$ is dense.
\end{claim} 

\begin{proof}
Let $x\in\mathbb{I}$ and $\delta>0$. We must now find $y\in B(x,\delta)\cap G_{h,\psi,i,m,\mathbf q}$. Let $t\in\N$ be such that $\frac{1}{N^t}<\delta$. We can clearly choose a positive integer $s\in\N$ and $z_1,...,z_s\,\in\{0,1,...,N-1\}$ such that if $z =\frac{z_1}{N}+\frac{z_2}{N^2}+...+\frac{z_s}{N^s}$ then $\mathbf \Pi(z, s) = \mathbf q$. Let 
\[
y = \frac{d_1(x)}{N}+...+\frac{d_t(x)}{N^t}+\sum_{d=0}^\infty\left( \frac{z_1}{N^{t+d s+1}}+\frac{z_2}{N^{t+d s+2}}+...+\frac{z_s}{N^{t+d s+s}} \right)
.\]
Then $y\in B(x,\delta)$ (as $y$ has its first $t$ digits the same as $x$). 

Next we show that $y\in G_{h,\psi,i,m,\mathbf q}$. All $z_i$'s cannot be $0$ or $N-1$, because we excluded the vectors $(1, 0,...,0),\,(0,...,0,1)$. Therefore, $y$ has a non-terminating $N$-adic expansion. Choose $j$ big enough such that
\[
\frac{j}{\psi(j)}<\frac{1}{h}\;\;\;\text{  and  }\;\;\; j\geq\max\left\{ i,t,hN\max_{0\leq\ell\leq N-1}\left\{ {N_\ell (z,s)\left(2+\frac{t}{s}\right)+N_\ell (y,t)} \right\} \right\} 
\]
where $N_\ell (x, n) = |\{1\leq k\leq n \; : \; d_k(x) = \ell \}|$. Fix a positive integer $n$ with $j < n < \psi^{m}(j)$ and observe that we can find integers $r$ and $b_\ell$ with $0\leq r < s$ and $0\leq b_\ell < N_\ell(z, s)$, such that $n = t + \left[\frac{n-t}{s}\right] s+ r$, and $N_\ell(y,n) = N_\ell(y, t) +\left[\frac{n-t}{s}\right]N_\ell(z,s)+b_\ell$. We now have,
\begin{align*}
\left\| \mathbf \Pi(y,n)-\mathbf q \right\|_1&=\sum_{\ell=0}^{N-1}  \left|  \frac{N_\ell(y,n)}{n}-\frac{n\frac{N_\ell(z,s)}{s}}{n}  \right|    \\
%%%5
&= \sum_{\ell=0}^{N-1} \left|  \frac{N_\ell(y,t)+[\frac{n-t}{s}]N_\ell(z,s)+b_\ell}{n}-\frac{(t+[\frac{n-t}{s}]s+r)\frac{N_\ell(z,s)}{s}}{n}  \right|      \\
%%%
&\leq \sum_{\ell=0}^{N-1}  \left(  \left|  \frac{[\frac{n-t}{s}]N_\ell(z,s)-\frac{N_\ell(z,s)}{s}(t+r)-N_\ell(z,s)[
\frac{n-t}{s}]}{n}  \right| +\frac{N_\ell(y,t)+b_\ell}{n}  \right) 
       \\
%%%
&<  \sum_{\ell=0}^{N-1} \left(  \frac{\frac{N_\ell(z,s)}{s}(t+r)+N_\ell(y,t)+N_\ell(z,s)}{n} \right)     \\
%%%
&< \sum_{\ell=0}^{N-1}  \left( \frac{N_\ell(z,s)(2+\frac{t}{s})+N_\ell (y,t)}{j}  \right) \leq \frac{1}{h}.   
%%%
\end{align*}
This shows that $y \in G_{h,\psi,i,m,\mathbf q}$ and completes the proof of Claim \ref{dense}.
\end{proof}

It follows from Claim \ref{open} and Claim \ref{dense}, that $A_{h,\psi}$ is the countable intersection of open and dense sets, and hence residual. This completes the proof of Lemma \ref{family}.
\end{proof}
%%%%%%%%%%%%%%%%%5
\begin{lemma}\label{extreme}
For any regular linear transformation $T$ and any $h\in\N$ we can:
\begin{enumerate}
    \item construct an increasing function $\psi_{h,T}:\N\mapsto\N$ such that $\psi(n)>n^2$ for all $n\in\N$ and,
    \item show that for some $M$ depending on $T$, we have that for each\[
    x\in A_{6Mh,\psi_{h,T}}\text{ the sequence $\left(\mathbf\Pi^T(x,m)\right)_m$ has property $P_{h,\psi_{h,T}}$.}
    \]
\end{enumerate}

\end{lemma} 

%%%%%%%%%%%%%%%
%%%%%%%%%%%%%

Before proceeding to prove Lemma \ref{extreme}, let us first recall a very useful characterisation result for \emph{regular} linear transformations. Given a linear transformation $T$ corresponding to an $\N\times\N$ real valued matrix $\left[c_{m,n}\right]$ and a sequence $(s_n)_n$ we are interested in the behaviour of its averaged version, that is the sequence whose $m$-th term is given by $t_m=\sum_{n\geq1} c_{m,n} s_n$. 
%%%%%%%%%%%%%%%%%%%5

\begin{proposition}[\cite{hard}]\label{regular}
A linear transformation $T=\left[c_{m,n}\right]_{\mathbb{N}\times\mathbb{N}}$ is regular if and only if the following conditions are satisfied:
\begin{enumerate}
\item there exists $M\in\N$ such that for all $m\in\mathbb{N}$ we have that
$\sum_{n=1}^{\infty}|c_{m,n}|<M$,
\item for all $n\in\mathbb{N}$, $\lim_{m\to\infty}c_{m,n}=0$ and
\item we have that $\lim_{m\rightarrow\infty}\sum_{n=1}^{\infty}c_{m,n}= 1$.
\end{enumerate}
\end{proposition}
%%%%%%%%%%%%%%5
%%%%%%%%%%%%%%%

\begin{proof}[Proof of Lemma \ref{extreme}]
Fix $h\in\N$ and $T=[c_{m,n}]$ a regular linear transformation. Firstly, we construct $\psi_{h,T}:\N\mapsto\N$. To do that recall that from Proposition \ref{regular} we can associate to $T$ a number $K\in\N$ and sequences of natural numbers $(M_k)_k$ and $(N_k)_k$ such that:
 \begin{enumerate}
     \item for each $m\in\N$ we have $\sum_{n\geq M_m} |c_{m,n}|\leq\frac{1}{12h}$,
     \item for each $n\in\N$ we have $|c_{m,n}|<\frac{1}{2^{n+10}h}$ whenever $m>N_n$,
     \item for $m>K$ we have $\left|1-\sum_{n\geq1}c_{m,n}    \right|<\frac{1}{2h}.$
 \end{enumerate}
Recall also $M\in\N$ the uniform bound of the absolute row sums of $T$. Having fixed these constants depending on the linear transformation $T$ and the natural number $h$, we proceed by defining our function $\psi$ by $\psi(0)=K+2$ and recursively for $n\geq1$:
$$
\psi_{h,T}(n):=M_n+N_n+n^2+\psi_{h,T}(n-1).
$$
This function $\psi=\psi_{h,T}$ is increasing and crucially $\psi(n)>\max\{n^2,M_n,N_n,K\}+1$ for all $n\geq0$. 
%%%%%%%%%%%%%%%%%5

Then, fix $i\in\N$, $m\in\N$ and $\mathbf q\in\mathbb D$. For $x\in A_{6Mh,\psi}$, the sequence $\left(\mathbf\Pi(x,n)\right)_n$ satisfies property $P_{6Mh,\psi}$. 
%%%%%%%%%%%%%%%%%%%%%%5
We can therefore find $j'\geq i$, with $\frac{j'}{\psi(j')}<\frac{1}{6Mh}$ and such that if $j'<n<\psi^{m+2}(j')$, then \begin{align}\label{step}
  \left\| \mathbf\Pi(x,n)-\mathbf q \right\|_1<\frac{1}{6Mh}.
\end{align}
Set $j=\psi(j')$. Then, clearly $j\geq i$ and $\frac{j}{\psi(j)}=\frac{\psi(j')}{\psi(\psi(j'))}<\frac{1}{\psi(j')}<\frac{j'}{\psi(j')}<\frac{1}{h}$. For all $j<r<\psi^m(j)$, that is $\psi(j')<r<\psi^{m+1}(j')$, we have
%%%%%%%%%%%%
\[
\left\| \mathbf \Pi^T(x,r)-\mathbf q \right\|_1\leq\left\| \sum_{n\geq1} c_{r,n} \left(\mathbf\Pi(x,n)-\mathbf q\right)  \right\|_1+\left\| \mathbf q (1-\sum_{n\geq1}c_{r,n}) \right\|_1\leq \sum_{n\geq1} |c_{r,n}| \left\| \mathbf\Pi(x,n)-\mathbf q \right\|_1+\frac{1}{2h}
\]
since $r>\psi(j')>K$ and $\norm{\mathbf q}_1=1$. Now we proceed by splitting our summation into three parts which we bound separately.
\begin{align*}
%%%%%%%%%%
 \sum_{n\geq 1} |c_{r,n}|\left\| \mathbf\Pi(x,n)-\mathbf q \right\|_1&=\sum_{n=1}^{j'} |c_{r,n}|\left\| \mathbf\Pi(x,n)-\mathbf q  \right\|_1\\
 %%%
 &+ \sum_{n\geq j'+1}^{\psi^{m+2}(j')-1} |c_{r,n}|\left\| \mathbf\Pi(x,n)-\mathbf q \right\|_1\\
%%%%%%%%%
&+ \sum_{n\geq \psi^{m+2}(j')} |c_{r,n}|\left\| \mathbf\Pi(x,n)-\mathbf q \right\|_1
\end{align*}
Firstly, using the fact that $r>\psi(j')$ we get that $r>\max\{N_1,...,N_{j'}\}$ and hence 
\begin{align*}
  \sum_{n\leq j'} |c_{r,n}|\left\| \mathbf\Pi(x,n)-\mathbf q   \right\|_1\leq 2\sum_{n\leq j'} |c_{r,n}|\leq \sum_{n\leq j'} \frac{1}{2^{n+9}h}\leq \frac{1}{6h}.
\end{align*}
Then using (\ref{step}) we can bound the second summand by
\[
\sum_{n\geq j'+1}^{\psi^{m+2}(j')-1} |c_{r,n}|\left\|  \mathbf\Pi(x,n)-\mathbf q   \right\|_1\leq \frac{1}{6 Mh} \sum_{n\geq 1} |c_{r,n}|\leq \frac{1}{6h}.
\]
Finally, for $n\geq\psi^{m+2}(j')=\psi(\psi^{m+1}(j'))>\psi(r)$ we have that $n>M_r$ hence
\[
\sum_{n\geq \psi^{m+2}(j')} |c_{r,n}|\left\| \mathbf\Pi(x,n)-\mathbf q \right\|_1\leq 2\sum_{n\geq M_r} |c_{r,n}|\leq \frac{1}{6h}.
\]
Putting all these together we get that for all $j<r<\psi^m(j)$  \[
\left\| \mathbf \Pi^T(x,r)-\mathbf q \right\|_1\leq \frac{1}{h}
\]
The choices of $i,m\in\N$ and $\mathbf q\in\mathbb D$ were arbitrary hence the sequence $\left(\mathbf\Pi^T(x,m)\right)_m$ has property $P_{h,\psi_{h,T}}$. This completes the proof of Lemma \ref{extreme}.

\end{proof}

%%%%%%%%%%%%%%%%%%%%%%%%%%%%%%5
% Recall the set $A$ given by\begin{align}\label{AA}
%   A=\bigcap_{T\in\mathcal{T}_\Q}A_T=\bigcap_{T\in\mathcal{T}_\Q}\bigcap_{h\in\N} A_{6Mh,\psi_{h,T}}. 
% \end{align}
%%%%%%%%%%%%%%%%%%%%%%%%%%%%

\begin{lemma}\label{subset} Let $T=[c_{m,n}]$ be a regular linear transformation and recall the set $A_T:=\bigcap_{h\in\N} A_{6Mh,\psi_{h,T}}$. Then
$$A_T\subset R_T:=\left \{ x\in[0, 1]\; :\; \text{The sequence $\left(\mathbf\Pi^T(x, m)\right)_m$ accumulates at every point of $\Delta_N$} \right\}.$$
\end{lemma} 
\begin{proof}
Let $x\in A_T$. To prove that $x\in R_T$, it suffices to show that each $\mathbf p\in \Delta_N$ is an accumulation point of $\left(\mathbf \Pi^T(x,m)\right)_m$. Therefore, we start by fixing $\mathbf p \in\Delta_N$ and $k\in \N$ and choose a vector $\mathbf q \in\mathbb D$ such that $\|\mathbf p-\mathbf q\|_1<\frac{1}{k}$. We claim that we can find $n_k>k$ such that
\begin{align}
    \left\|\mathbf q-\mathbf \Pi^T(x, n_k)\right\|_1\leq\frac{1}{k}.
\end{align}

We now prove the claim above. Indeed, for any for $x \in A_T$, by using Lemma \ref{extreme} we have that $(\mathbf \Pi^T(x,m))_m$ has property $P_{k,\psi_{k,T}}$. In particular, we can find $j \in N$ with $j\geq k$ and such that if $j < n < \psi_{k,T}(j)$ then
 $ \|\mathbf q-\mathbf \Pi^T(x, n)\|_1\leq\frac{1}{k}$. Hence if $n_k$ is any integer with $j < n_k < \psi_{k,T}(j)$ then \[  \|\mathbf q-\mathbf \Pi^T((x, n_k)\|_1\leq\frac{1}{k},\]
and in particular $n_k>j>k$ as required. Then we get that,
\begin{align}\label{convergence}
    \left\| \mathbf p -\mathbf \Pi^T(x,n_k)    \right\|_1\leq \left\|  \mathbf p- \mathbf q   \right\|_1 + \left\|  \mathbf q -\mathbf \Pi^T(x,n_k)    \right\|_1\leq \frac{2}{k}.
\end{align}

Since $n_k > k$, we can extract an increasing subsequence $(n_{k_u})_u$ of $(n_k)_k$. It now follows from (\ref{convergence})
that $\mathbf \Pi^T(x,n_{k_u})  \ra\mathbf p$ as $u\ra\infty$ with respect to $\|.\|_1$. Hence $\mathbf p$ is an accumulation point of
$(\mathbf \Pi^T(x,m))_m$. This completes the proof of Lemma \ref{subset}. 
\end{proof} 

We can now prove Theorem \ref{main}.
\begin{proof}[Proof of Theorem \ref{main}]
By Lemma \ref{family} it follows that for each $h\in\N$ the set $A_{6Mh,\psi_{h,T}}$ is residual in $\mathbb I$. Since $A_T:=\bigcap_{h\in\N} A_{6Mh,\psi_{h,T}}$ it follows that the set $A_T:=\bigcap{h\in\N}A_{6Mh,\psi_{h,T}}$ is residual in $\mathbb I$ as a countable intersection of residual sets. Then, by Lemma \ref{subset} since for each regular linear transformation $T$, the set $A_T$ is a subset of $R_T$ we get that $R_T$ is residual in $\mathbb I$. Finally, it can be easily seen that $[0, 1]\setminus \mathbb I$ is a countable union of nowhere dense sets and so $R_T$ is residual in $[0, 1]$.
\end{proof}

\begin{proof}[Proof of Corollary \ref{dimen}]
Assume for the sake of contradiction that $\dim_P R_T<1$. Recall that for $F\subseteq\R$ if we let $\overline{\dim}_B(F)$ denote the upper box dimension of $F$ then we have that
$$
\dim_P(F)=\inf\left\{ \sup_{n\in\N} \overline{\dim}_B(E_n)\; :\; E_1,E_2,...\subseteq\R\;\;\text{and}\;\; F\subseteq\bigcup_{n\in\N} E_n    \right\}.
$$
Then by assumption we can find a countable family of subsets $E_1,E_2,...$ of $\R$ with $R_T\subseteq\bigcup_n E_n$ such that $\sup_n \overline{\dim}_B(E_n)<1$. Since $R_T\subseteq[0,1]$ and $R_T\subseteq\bigcup_n E_n$ we get that $$
R_T\subseteq\bigcup_n\big( [0,1]\cap E_n  \big).
$$
To arrive at the desired contradiction we show that for all $n\in\N$ the set $[0,1]\cap E_n$ is nowhere dense in $[0,1]$ contradicting Theorem \ref{main}. Assume this is not the case. Then for some $n\in\N$ and some open set $U\subseteq \R$ we have that$$
[0,1]\cap U\neq \emptyset \;\;\;\;\text{and moreover}\;\;\;\;[0,1]\cap U\cap E_n\text{  is dense in  }[0,1]\cap U.
$$
In particular, we deduce that $[0,1]\cap U\subseteq \overline{[0,1]\cap U\cap E_n}$ and therefore:
\begin{align*}
    \overline{\dim}_B(E_n)&=\overline{\dim}_B(\overline{E_n})\\
    &\geq \overline{\dim}_B(\overline{[0,1]\cap U\cap E_n})\\
    &= \overline{\dim}_B(\overline{[0,1]\cap U})\\
    &\geq  \dim_P([0,1]\cap U)=1,
\end{align*}where we have used that $\dim_P([0,1]\cap U)=1$ for all open sets $U$ that intersect $[0,1]$. This is a contradiction since by assumption $\sup_n \overline{\dim}_B(E_n)<1$. This completes the proof.
\end{proof}

%%%%%%%%%%%%%%%%%%%%%%%%%%%%%%%%%%%%%%%%%%%%%%%5

\begin{ack} 
The author would like to thank Professor Lars Olsen for presenting him with this problem and Paolo Leonetti for useful discussions.
\end{ack} 
%%%%%%%%%%%%%%%%%%%%%%%%%%%%%%%%%%%%%%%%%%%%%%%%%%%%%%%%%%%%%%%%%%%%%%%%%%%%%%%%%

\end{document}